\documentclass[12pt,reqno]{amsart}
\usepackage{amsthm}
\usepackage{amsthm}
\usepackage{amsaddr}
\usepackage[x11names]{xcolor}
\usepackage[linkcolor=Blue2,citecolor=red,colorlinks]{hyperref}
\usepackage[capitalize,nameinlink]{cleveref}
\usepackage{amsthm}
\usepackage{mathrsfs}

\usepackage{amsmath}
\author [Bhattacharyya, Bhowmik]{Tirthankar Bhattacharyya, Mainak Bhowmik}
\address{Department of Mathematics, 	Indian Institute of Science, 
	Bangalore 560012, India}
\email{tirtha@iisc.ac.in;  mainakb@iisc.ac.in}
\author [Kumar]{Poornendu Kumar}
\address{Department of Mathematics, University of Manitoba, Winnipeg R3T 2N2, Canada }
\email{poornendu.kumar@umanitoba.ca}
\usepackage{color, bm, amscd, tikz-cd}
\newcommand{\cA}{{\mathcal A}}
\newcommand{\cB}{{\mathcal B}}
\newcommand{\cC}{{\mathcal C}}

\newcommand{\cH}{{\mathcal H}}

\newcommand{\cP}{{\mathcal P}}
\newcommand{\cQ}{{\mathcal Q}}

\newcommand{\cW}{{\mathcal W}}

\newcommand{\bT}{{\mathbb{T}}}

\newcommand{\bD}{{\mathbb D}}
\newcommand{\bG}{{\mathbb G}}
\newcommand{\bC}{{\mathbb C}}
\newcommand{\bN}{{\mathbb N}}

\newtheorem{thm}{Theorem}[section]

\newtheorem{lemma}[thm]{Lemma}

\theoremstyle{plain}
\newtheorem{definition}[thm]{Definition}
\newtheorem{remark}[thm]{Remark}

\theoremstyle{definition}
\newtheorem{example}[thm]{Example}

\numberwithin{equation}{section}

\newcommand\blfootnote[1]{%
	\begingroup
	\renewcommand\thefootnote{}\footnote{#1}%
	\addtocounter{footnote}{-1}%
	\endgroup
}

\begin{document}
	\title{Herglotz's representation and Carath\'eodory's approximation}
	\maketitle
	{\blfootnote{{2020 {\em Mathematics Subject Classification}: 30E10, 32A25, 32A70, 47A56.\\
				{\em Key words and phrase}:  Herglotz's representation, Krein-Milman Theorem,  Multi-connected domains, Carath\'eodory's approximation, Inner functions, Szeg\"o Kernels, Harmonic measures,  Hardy space.}}}

	\maketitle
	\begin{abstract}
		Herglotz's representation of holomorphic functions with positive real part and Carath\'eodory's theorem on approximation by inner functions are two well-known classical results in the theory of holomorphic functions on the unit disc. We show that they are equivalent.
		
 On a multi-connected domain $\Omega$, a version of Heglotz's representation is known. Carath\'eodory's  approximation was not known. We formulate and prove it and then show that it is equivalent to the known form of Herglotz's representation. Additionally, it also enables us to prove a new Heglotz's representation in the style of Koranyi and Pukanszky. Of particular interest is the fact that the scaling technique of the disc  is replaced by Carath\'eodory's approximation theorem while proving this new form of Herglotz's representation. 
 
  Carath\'eodory's approximation theorem is also proved for matrix-valued functions on a multi-connected domain.  
	
	\end{abstract}

\vspace{5mm}

\noindent \textbf{Acknowledgement:} \\
	{\footnotesize{T. Bhattacharyya is supported by the J C Bose Fellowship  JCB/2021/000041 of SERB, M. Bhowmik is supported by the Prime Minister's Research Fellowship PM/MHRD-21-1274.03 and P. Kumar is supported by a PIMS postdoctoral fellowship.  This research is supported by the DST FIST program-2021 [TPN-700661]. \\
}}

\newpage

	\section{Introduction}
	
	The representation
	\begin{equation}
		f(z)= i \textup{Im}f(0) + \int_\bT \frac{\alpha+z}{\alpha-z}d\mu(\alpha) \label{Her}\\
	\end{equation}
	of a holomorphic function defined on the open unit disc $\mathbb D$ having non-negative real part is due to Herglotz \cite{Her}. The equation \eqref{Her} connects with the theory of operators on a Hilbert space in an obvious way. By considering $L^2(\mu)$ and the normal operator $U$ (with spectrum contained in the unit circle and hence a unitary operator in this case) of multiplication by the co-ordinate $z$, the equation above takes the form
	\begin{equation}
		f(z) = i \textup{Im}f(0) + (\langle (U + z)(U - z)^{-1} 1 , 1 \rangle) \textup{Re}f(0) \label{RF}
	\end{equation}
	where $1$ is the constant function $1$.  We shall call \eqref{Her} an {\em integral representation} whereas \eqref{RF} is known as a {\em realization formula}.
	
	\begin{definition} \label{Hclass}
		Given a domain $\Omega$ in $\mathbb C^d$ and a Hilbert space $\mathcal H$, the Herglotz class $\mathfrak H(\Omega, \mathcal B (\mathcal H))$ consists of all $\mathcal B(\mathcal H)$-valued functions defined on $\Omega$ which are holomorphic and have non-negative real part. For a fixed $\mathbf{z}_0 \in \Omega$, we shall denote by $\mathfrak{H}_{\mathbf{z}_0}(\Omega, \mathcal{B}(\mathcal{H}))$ the class of all $f \in \mathfrak{H}(\Omega, \mathcal{B}({\mathcal{H}}))$ satisfying $f(\mathbf{z}_0)= I$. When $\mathcal H = \mathbb C$, we shall denote these classes by $\mathfrak H(\Omega)$ and $\mathfrak{H}_{\mathbf{z}_0}(\Omega)$ respectively.
	\end{definition}
	
	Expectedly, the two forms above have been generalized in two different directions. See \cite{Kor-Puk} and \cite{AHR} where the authors derived integral representations for $\mathfrak{H}(\Omega)$ which hold in a large class of domains including the annulus, the polydisc and the Euclidean unit ball. On the other hand, see \cite{AY, Ball-V}, where realization formulae were obtained for a subclass of $\mathfrak{H}(\mathbb D^d)$. This note develops the first point of view, i.e., the integral representation in tandem with a classical theorem of Carath\'eodory. 
	
	The power of Herglotz's Theorem is well-known - see for example \cite{KBS} for a proof of the spectral theorem for unitary operators using it or \cite{AHR} for a proof of the von Neumann inequality using it. It plays a role in proving that the rational dilation fails in a triply connected domain, see \cite{DM}. Herglotz's Theorem has applications to de Branges-Rovnyak spaces (see the seminal work \cite{Sarason}) as well as to two-phase composite materials, where the scalar-valued Herglotz functions correspond to the effective properties of the composite materials and the matrix-valued Herglotz functions are applied to study the permeability tensor of a porous material, see \cite{OL}. 
	
	In spite of many applications of Herglotz's representation, it has never been applied to generalize a landmark result of Carathe\'odory who, in his study of holomorphic self maps on the open unit disc, proved the following theorem, see Section 284 in \cite{C}. 
	\begin{thm}[Carath\'eodory] \label{CR}
		Any holomorphic function $\varphi: \mathbb{D} \to \overline{\mathbb{D}}$ can be approximated (uniformly on compacta in $\mathbb{D}$) by rational inner functions, i.e., rational functions whose poles are outside $\overline{\mathbb D}$ and which assume unimodular values on the unit circle. \end{thm}
	A proof, which seamlessly extends itself to matrix-valued functions, of this theorem can be done via Pick-Nevanlinna interpolation. In a recent work \cite{B_J_K}, the use of unitary dilation of contractions establishes the same result. Herglotz's representation for matrix-valued functions gives yet another proof. This new proof extends to operator-valued functions to give a version of Carath\'eodory's result. This is in Section \ref{C}. This section also shows that for scalar-valued functions, Herglotz's theorem is equivalent to Carath\'eodory's theorem in the disc.  
	
	Koranyi and Pukanszky noticed that the Szeg\"o kernel on $\bD$ given by 
	$$S_{\bD}(z,\zeta)=\frac{1}{1-z\overline{\zeta}}\ \text{for}\ z\in \bD \ \text{and}\ \zeta \in \bT$$ 
	can be used to rewrite Herglotz's representation \eqref{Her} as 
	$$f(z)= i\operatorname{Im} f(0)+ \int_{\mathbb{T}} \left(2S_{\mathbb{D}}(z,\zeta)-1 \right)d\mu(\zeta). $$
	Taking cue from here, they proved an integral representation for $\mathfrak{H}(\mathbb{D}^d)$ functions generalizing Herglotz's representation  to the polydisc in \cite{Kor-Puk}. It says that {\em a function $f \in \mathfrak{H}(\mathbb{D}^d)$ if and only if it is of the following form:}
	\begin{align*}\label{koranyi}
		f(\mathbf{z})= i\operatorname{Im} f(0)+ \int_{\mathbb{T}^d} \left(2S_{\mathbb{D}^d}(\mathbf{z},\mathbf{u})-1 \right)d\mu(\mathbf{u})
	\end{align*}
	where $S_{\mathbb{D}^d}$ is the Szeg\"o kernel on the polydisc and  $\mu$  is a non-negative measure satisfying
	\begin{equation} \label{DeterminingKP} \int_{\bT^d} u_1^{n_1} \cdots u_d^{n_d} d\mu(\mathbf u) =0
	\end{equation}
	unless $n_j \geq 0$ for all $j=1, \cdots, d$ or  $n_j \leq 0$ for all $j=1, \cdots, d$. 
	This integral representation is applicable mutatis mutandis for bounded homogeneous domains which are starlike, circular, and whose isotropy groups are linear and transitive on the Bergman-Shilov boundary.  Generalizations of these results go well beyond the classical Cartan domains where the Koranyi-Pukanszky formula is known. This is the content of \cite{Dautov}.
	
	The condition \eqref{DeterminingKP} hints that instead of the full dual of $C(\mathbb T^d)$, a quotient is found useful. Indeed, this is a new feature that arises as one moves away from the unit disc. Definition \ref{deter} and the results following it bring out the full clarity. 
	
	When one replaces the disc by a multi-connected domain in the plane with the aim to obtain a representation in the style of Herglotz, but having a suitable kernel in the manner of the Koranyi and Pukanszky, a natural challenge presents itself. A crucial step in the proof of Herglotz's Theorem in the disc is a scaling technique. This produces a class of approximating functions which gives rise to a class of measures. One then applies Helly's selection principle. In a multi-connected domain, the scaling technique fails. But we prove a generalization of the classical theorem of Carath\'eodory mentioned above. This is a new result and that comes to the rescue. So,  there is a reversal of roles. This pleasantly surprising contribution of the generalized theorem of Carath\'eodory to prove the generalization of Herglotz's Theorem for multi-connected domains is done in Section \ref{multi}. In this case too, Carath\'eodory's approximation is equivalent to a version of Herglotz's representation proved earlier in \cite{AHR}.

The measure given by Herglotz's theorem (for scalar-valued functions) in the unit disc is unique. In general, the uniqueness is lost. The question of whether the measure obtained from Herglotz's theorem is unique is related to the question of whether Herglotz's representation can be obtained for operator valued functions. In Section \ref{Opvalued}, we show that the uniqueness of the measure in Herglotz's representation of scalar-valued functions always leads to Herglotz's representation for operator-valued functions. 

 The standard technique of extending Herglotz's representation from scalar-valued functions to operator-valued functions requires the uniqeness of the representing measure in Herglotz's representation for a scalar-valued function. In spite of the lack of uniqueness in a multi-connected domain, a version of Herglotz's representation for matrix valued functions has been proven in \cite{BH}. Leveraging the ideas from there, we prove Carath\'eodory's approximation theorem for a holomorphic function which takes values in $N\times N$ matrices and has norm no larger than $1$. Such functions are approximated uniformly on compact subsets of the multi-connected domain $\Omega$ by {\em inner functions}, i.e., functions which are holomorphic, take values in $N\times N$ matrices and are unitary-valued almost everywhere on the boundary $\partial \Omega$.

	\section{Carath\'eodory's theorem for operator-valued functions on the disc} \label{C}
	
	\subsection{Herglotz's representation implies Carath\'eodory's approximation}

Proofs of Carath\'eodory's theorem for matrix-valued functions are available from various perspectives, including the function-theoretic \cite{C, Rudin-polydisc} and operator-theoretic \cite{B_J_K} viewpoints. Here we demonstrate a new proof of Carathéodory's theorem for matrix-valued functions on $\bD$, from the perspective of operator algebras which gives a version of Carathéodory's theorem for operator-valued functions, hitherto unknown.
	
		
		
	
	\begin{thm}\label{M-Cara}
		A holomorphic map $f:\bD\rightarrow M_N(\bC)$ satisfying $\|f(z)\|\leq 1$ for all $z\in\bD$ can be approximated uniformly on compacta in $\bD$ by  a sequence of $M_N(\bC)$-valued rational inner functions i.e., rational functions  with poles off $\overline{\mathbb D}$ whose boundary values (on the unit circle) are unitary matrices.
	\end{thm}
	
	\begin{proof}

		By a straightforward application of the maximum norm principle, either $\|f(z)\|<1$ for all $z \in\mathbb{D}$, or $\|f(z)\|= 1$ for all $z\in\mathbb{D}$. 

In the case when $\|f(z)\|<1$ for all $z \in\mathbb{D}$, define a holomorphic map $\theta: \bD\rightarrow M_N(\bC)$  by 
		\begin{align*}
			\theta(z)= \left(I-f(z)\right)\left(I+f(z)\right)^{-1}.
		\end{align*}
		Since the real part of $\theta$ is non-negative, by Herglotz's Theorem (for matrix-valued functions), there exists a positive $M_N(\bC)$-valued bounded regular Borel measure $E$ on $\mathbb{T}$ with $E(\bT)= \operatorname{Re}\theta(0)$ such that
		\begin{align}\label{theta}
			\theta(z)= i\operatorname{Im}\theta(0) + \int_{\mathbb{T}} \frac{\alpha+z}{\alpha-z}dE(\alpha).
\end{align}
		See \cite{GKMT}. This measure induces a (completely) positive map $\Theta_{E}: C(\bT)\rightarrow M_N(\bC)$:
		$$ \Theta_{E}(h)= \int_{\mathbb{T}} h dE, \ \text{for}\ f\in C(\mathbb{T}).$$

		There is a one to one correspondence between the set of all positive $M_N(\bC)$-valued bounded regular Borel measures $\Lambda$ on $\mathbb{T}$ with $\Lambda(\bT)= \operatorname{Re}\theta(0)$ and the set of (completely) positive maps $\Phi$ from $ C(\bT)$ to $M_N(\bC)$, with $\Phi(1)= \operatorname{Re}\theta(0)$, denoted by $CP\left(C(\mathbb{T}), M_N(\mathbb{C}) \right)$. It is compact in the bounded weak (BW) topology, see \cite{Arveson1} for details on this topology. The extreme points of $CP\left(C(\mathbb{T}), M_N(\bC) \right)$ are the positive maps of the form
		$$\Phi (h)= h(\alpha_1)A_1 + \cdots + h(\alpha_n)A_n, \ h\in C(\mathbb{T}), $$
		for some $\alpha_1, \cdots, \alpha_n \in \mathbb{T}$ and positive $N\times N$ matrices $A_1,\cdots, A_n$ satisfying
		\begin{enumerate}
			\item[(i)] $A_1+\cdots + A_n = \operatorname{Re}\theta(0)$ and
			\item[(ii)] $\left\lbrace [A_1(\mathbb{C}^N), \cdots, A_n(\mathbb{C}^N)]  \right\rbrace$ is a weakly independent family of subspaces.
		\end{enumerate}
		See Theorem 1.4.10 in \cite{Arveson1}. An application of the Krein-Milman theorem enables us to conclude that $CP\left(C(\mathbb{T}), M_N(\bC) \right)$ is the BW-closure of the convex hull of the extreme points. It is also known that the BW-topology on $CP\left(C(\mathbb{T}), M_N(\mathbb{C}) \right)$ is metrizable and hence is sequentially BW-compact.We shall apply this below.

		Get a sequence $\{\beta_n\}$ in the convex hull of the extreme points such that $\beta_n (h) \rightarrow \Theta_E(h)$ in operator norm (the function is matrix-valued), for every $h\in C(\bT)$. Given a compact subset $K$ of $\mathbb{D}$, for each $z\in K$, the function 
		$$h(\alpha)= \frac{\alpha + z}{\alpha - z},$$
		is continuous on $\mathbb{T}$. Note that 
		$$\theta_n(z):= i\operatorname{Im}\theta(0) + \beta_n(h)  \rightarrow  \theta(z) \ \text{in operator norm}.$$
It is known that
		$$ \beta_n(g) = \sum_{k=1}^ {l(n)} \sum_{j=1}^{r(k)} \lambda_k^{(n)}  g(\alpha_j^{(k)}) A_j^{(k)}\text{ for } g \in C(\mathbb T),$$
		for certain points $\alpha_j^{(k)}$ in $\mathbb{T}$, positive scalars $\lambda_k^{(n)}$, satisfying $\sum_{k=1}^{l(n)} \lambda_k^{(n)} =1$ and  $N\times N$ positive matrices $ A_j^{(k)}$ with $\sum_{j=1}^{r(k)} A_j^{(k)}= \operatorname{Re}\theta(0)$ for each $k$.  Thus, 
		$$\theta_n(z) = i\operatorname{Im}\theta(0) + \sum_{k=1}^ {l(n)} \sum_{j=1}^{r(k)} \lambda_k^{(n)}  \frac{\alpha_j^{(k)} +z}{\alpha_j^{(k)} -z} A_j^{(k)}.$$
		This structure immediately implies that $\theta_n$ is a rational Herglotz class function with the additional property of $\operatorname{Re}\theta_n(\zeta)=0$ almost everywhere on the unit circle $\mathbb{T}$. Considering the inverse Cayley transform 
		$$
		f_n(z)=(I-\theta_n(z))(I+\theta_n(z))^{-1}
		$$
		of $\theta_n$ we observe that $\{f_n\}_n$ is a sequence of rational inner functions. An application of Montel's theorem then completes the proof.
		
		On the other hand if $\|f(z)\|=1$ for every $z\in \bD$, by Theorem $4$ in \cite{Con}, we can find two $N\times N$ unitary matrices $U_1$ and $U_2$ such that 
		\begin{align*} 
			f(z)= U_1
			\begin{bmatrix}
				1 & 0\\ 0 & g(z)
			\end{bmatrix}
			U_2 \ \text{for}\ z\in \bD
		\end{align*}
		for some $M_{N-1}(\bC)$-valued holomorphic function $g$ on $\bD$ such that $\|g(z)\| \leq 1$ for every $z\in \bD$. If $\|g(z)\|<1$ for each $z$ then we approximate $g$ by a sequence $\{g_n\}$ of $M_{N-1}(\bC)$-valued rational inner functions and hence the rational inner functions  
		$$f_n = U_1 \begin{bmatrix} 1 & 0 \\ 0 & g_n \end{bmatrix} U_2 $$ will converge to $f$ uniformly on compacta in $\bD$. Otherwise, we repeat the process for $g$. This process stops in finitely many steps.
	\end{proof}
	
	The proof above adapts itself to some extent in the infinite-dimensional case.
	
	\begin{thm}\label{O-Cara}
		A holomorphic map $f : \mathbb{D} \rightarrow \cB(\cH)$ satisfying $\operatorname{Re} f(z)> 0$ for all $z \in \mathbb{D}$ can be approximated (uniformly on compacta in $\mathbb{D}$) in the weak operator topology (WOT) by a net $\{f_\lambda \}$ of Herglotz class functions which also satisfy $\operatorname{Re}f_\lambda(\zeta)=0$ almost everywhere on the unit circle $\mathbb{T}$.
	\end{thm}
	\begin{proof}
		After getting a positive $\cB(\cH)$-valued bounded regular Borel measure $E$ on $\mathbb{T}$ as in \eqref{theta}, an application of Theorem 6.15 of \cite{Banerjee-Bhat-Kumar} (actually a slight variant of this theorem) produces a net $\{E_{\lambda}\}$ of $C^*$-convex combinations of Dirac spectral measures  such that
		\begin{align*}
			E_{\lambda}\xrightarrow{BW-topology} E \quad \text{ (as net) }.
		\end{align*}
		Each $E_{\lambda}(\cdot)$ is of the form
		$\sum_{j=1}^{N(\lambda)}T_j^*\delta_{\alpha_j}(\cdot)T_j$ with $\sum T_j^*T_j= \operatorname{Re}f(0)$. Define
		$$f_\lambda(z)=i \operatorname{Im}f(0)+ \int_{\mathbb{T}} \frac{\alpha+z}{\alpha-z}dE_\lambda(\alpha) .$$
		For each $z\in\mathbb{D}$, the map $\alpha\mapsto \frac{\alpha+z}{\alpha-z}$ is continuous on $\mathbb{T}$. Therefore, 
		\begin{align*}
			\int_{\mathbb{T}} \frac{\alpha+z}{\alpha-z}dE_\lambda(\alpha)&= \int_{\mathbb{T}} \frac{\alpha+z}{\alpha-z}\sum_{j=1}^{N(\lambda)}T_j^*\delta_{\alpha_j}(\cdot)T_j= \sum_{j=1}^{N(\lambda)}\frac{\alpha_j+z}{\alpha_j-z}T_j^*T_j.
		\end{align*}
		Fix, $h,k \in \cH$. For each $\lambda$ we define, $$\psi_\lambda^{h,k} (z) = \langle f_\lambda(z) h, k \rangle \ \text{for}\ z\in \bD.$$
		 Clearly, $\psi_\lambda^{h,k} $ is holomorphic in $\bD$. Recall from the previous theorem that for $E_\lambda$ we have a completely positive map $\Theta_{E_\lambda}$ on $C(\mathbb{T})$ with norm $\| \Theta_{E_\lambda}(1)\|= \|\operatorname{Re}f(0) \|$ such that
		 $$\Theta_{E_\lambda} \left( \frac{\alpha+z}{\alpha-z} \right)= 	\int_{\mathbb{T}} \frac{\alpha+z}{\alpha-z}dE_\lambda(\alpha) $$ for every $z\in \bD$.
		 Now, for $0<r<1$ consider $z\in \bD_r=\{w: |w|\le r\}$. Then
		 \begin{align*}
		 |\psi_\lambda^{h,k}(z)| & \le \|\operatorname{Im}f(0)\| \|h\| \|k\|+ |\langle \Theta_{E_\lambda} \left( \frac{\alpha+z}{\alpha-z} \right)h, k  \rangle| \\
		 & \le \left( \|\operatorname{Im}f(0)\|  + \frac{1+r}{1-r}   \|\operatorname{Re}f(0) \| \right) \|h\| \|k\|.
		 \end{align*}
		This implies that $\{\psi_\lambda^{h,k}\}_\lambda$ is uniformly bounded on $\bD_r$ for each $r$. Hence $\{\psi_\lambda^{h,k}\}_\lambda$ is uniformly bounded on compact subsets of $\bD$.
		
		Using this information along with the Cauchy integral formula, one can prove that the family $\{\psi_\lambda^{h,k}\}_\lambda$ is equicontinuous on compact subsets of $\bD$. We also have the pointwise convergence of $\{\psi_\lambda^{h,k}\}_\lambda$, i.e., 
		$$ \psi_\lambda^{h,k}(z) \xrightarrow{\text{as net}} \langle f(z)h, k \rangle $$ 
for every $z\in \bD$. An $\epsilon/3$-argument using uniform continuity of $z \mapsto \langle f(z)h, k \rangle $ on any compact subsets of $\bD$ now shows that
		$$\psi_\lambda^{h,k}(z) \xrightarrow{\text{as a net}} \langle f(z)h, k \rangle \ \text{uniformly on compacta in }\ \bD.$$
		The functions $\{f_{\lambda}\}$ are in the Herglotz class because 
		\begin{align*}
			\operatorname{Re} f_{\lambda}(z)&=\sum_{j=1}^{N(\lambda)}\operatorname{Re}\left[\frac{\alpha_j+z}{\alpha_j-z}T_j^*T_j\right]=\sum_{j=1}^{N(\lambda)}\operatorname{Re}\left[\frac{\alpha_j+z}{\alpha_j-z}\right]T_j^*T_j
		\end{align*}
		which is non-negative for $z\in \bD$. Also, it is zero on $\bT$ except for finitely many points. Hence the net $\{f_\lambda\}$ is the desired net of functions.
	\end{proof}
\subsection{Carath\'eodory's approximation implies Herglotz's representation}
We close this section by showing that in the case of scalar-valued functions, the Herglotz's representation can be proved using Carath\'eodory's approximation theorem.

So, we start with  a holomorphic map $f : \mathbb{D} \rightarrow \mathbb{C}$ such that $\operatorname{Re} f(z)\geq 0$ for all $z \in \mathbb{D}$.  If $\operatorname{Re}f(0)=0$ then by the open mapping theorem $f$ must be the constant function $i \operatorname{Im}f(0)$. Thus, assume $\operatorname{Re}f(0) > 0$. Consider the function 
	$$ g(z)= \frac{1}{\operatorname{Re}f(0)} (f(z)-i \operatorname{Im}f(0)). $$ 
Then, $g(0) =1$ and the real part of $g$ is non-negative. So, without loss of generality assume that $f(0)=1$.
Then the Cayley transform $\theta= \left(1-f\right)\left(1+f\right)^{-1}$ of $f$ is holomorphic in $\bD$ and $|\theta(z)|\leq 1$ for all $z\in\mathbb{D}$. Also, $\theta(0)=0$. Invoke Carath\'eodory's Theorem \ref{CR} to get a sequence of finite Blaschke products $\{B_n\}$ with $B_n(0)=0$ (see \cite[Theorem 2.1, Chapter I]{Garnett}) such that $B_n\rightarrow \theta$ uniformly on compacta in $\bD$. Let $\psi_n$ denote the inverse Cayley transform of the finite Blaschke product $B_n$. It is known (\cite[Lemma 5.1]{DU}) that $\psi_n$ is of the form 
$$
\sum_{j=1}^{r(n)} \lambda_j \frac{\alpha_j + z}{\alpha_j - z} = \int_\bT \frac{\alpha + z}{\alpha - z} d\mu_n (\alpha)
$$
where the measure $\mu_n$ is of the form: $\mu_n =\sum_{j=1}^{r(n)} \lambda_j \delta_{\alpha_j} $ for finitely many points $\alpha_j$ in $\bT$ and non-negative scalars $\lambda_j$ with $\sum_{j=1}^{r(n)} \lambda_j =1 $. Since $\{\mu_n\}_{n\geq 1}$ is a sequence of regular Borel probability measures on $\bT$, we can extract a subsequence $\{\mu_{n_k}\}$ such that $\mu_{n_k} \xrightarrow{weak^*} \mu$  for some regular Borel probability measure $\mu$. Therefore 
\begin{align*}
	\psi_{n_k} (z)= \int_\bT \frac{\alpha + z}{\alpha - z} d\mu_{n_k} (\alpha) \rightarrow \int_\bT \frac{\alpha + z}{\alpha - z} d\mu (\alpha).
\end{align*}
But $\psi_{n_k}(z)$ converges to $f(z)$ for each $z$. Hence we have 
$$
f(z)= \int_\bT \frac{\alpha + z}{\alpha - z} d\mu (\alpha).
$$ 
This is Herglotz's representation for $f$.

	\section{Carath\'eodory and Herglotz in multi-connected domains} \label{multi}
	
Corresponding to a bounded multi-connected planar domain $\Omega$ with boundary consisting of $(m+1)$ many disjoint simple analytic closed curves $\partial_0, \dots, \partial_m$, an $m+1$-dimensional torus $T_\Omega$ defined by 
	$$T_{\Omega}=\partial_0\times \dots\times \partial_m$$
will play a crucial role. Let $z_0$ be a point in $\Omega$. One of the ways Herglotz's theorem is proved in the disc case is through a convexity argument by first showing that $z \mapsto \frac{\alpha +z }{\alpha -z}, \alpha \in \bT$ are the extreme points of $\mathfrak{H}_0(\bD)$ and then applying a Krein-Milman type theorem. The extreme points of $\mathfrak{H}_{z_0}(\Omega)$ have been identified in \cite{Forelli}. It was natural that this would lead to a certain form of Herglotz's theorem and did in \cite{AHR}. Our aim here is to prove a Koranyi-Pukanszky type theorem. To that end, we need to recall the family $\{\Gamma_\alpha : \alpha \in T_\Omega\}$ of extreme points of $\mathfrak{H}_{z_0}(\Omega)$. 

These functions have been studied in \cite{Fisher-Khavinson, Forelli, Grunsky, Heins}. We shall recall from \cite{AHR}. For a fixed $z_0 \in \Omega$, consider the harmonic measure $\omega_{z_0}$ on $\partial \Omega$ for $z_0$. To elaborate, any $f \in C_{\mathbb{R}}(\partial \Omega)$ has a unique harmonic extension $u_f$ in $\Omega$. Consider the bounded linear functional $f \mapsto u_f(z_0)$ on $C_{\mathbb{R}}(\partial \Omega)$. The measure $\omega_{z_0}$ is the dual element which implements this linear functional. Let $P_z(\lambda)$ be the Poisson kernel which produces $u_f$ from $f$. In a multi-connected domain, a harmonic function need not have a single-valued harmonic conjugate in general. However, there is an $(m+1)$-tuple of positive weights $(w_0(\alpha),\dots, w_m(\alpha))$ indexed by $\alpha \in T_\Omega$ and adding upto $1$ such that 
$$
w_0(\alpha)P_z(\alpha_0)+ w_1(\alpha)P_z(\alpha_1)+\dots + w_m(\alpha)P_z(\alpha_m)
$$
has a single-valued harmonic conjugate. An extreme point $\Gamma_\alpha$ is the unique holomorphic functions on $\Omega$ whose real part is the function above and whose imaginary part vanishes at $z_0$. It is straightforward that $\Gamma_\alpha(z_0)=1$ for all $\alpha \in T_\Omega$ and 
$\operatorname{Re}\Gamma_\alpha (z)$ is positive for $z\in \Omega$ and for all $\alpha \in T_\Omega$.

Recall from the proof of Herglotz's representation in the disc case that a scaling is required. Since scaling is not available in a multi-connected domain, we need to find a different technique. This is what we develop first.


	\subsection{Carath\'eodory's theorem in multi-connected domains} \label{Cara}
 A complex-valued bounded holomorphic function $f$ on $\Omega$ is said to be \textit{inner} if the boundary values (which exist almost everywhere in $\partial \Omega$) of $f$ are unimodular $\omega_{z_0}$-almost everywhere on $\partial \Omega$. See \cite{BFKS} for a related but different notion of inner functions.

 We are interested in a very specific type of inner functions in $\cA(\Omega)$, viz., those belonging to the following class of functions
 \begin{align}\label{Inner-Mul}
 \mathscr{RI}(\Omega)=\left\{\frac{1-\theta}{1+\theta}: \theta= ia+ b \sum_{j=1}^{n} c_j \Gamma_{\alpha_j}\text { where } b >0, a\in\mathbb{R}, n\in \mathbb{N} \text { and } \sum_{j=1}^{n} c_j =1 \text{ with } c_j\geq 0  \right\}.	
 \end{align}
Every member of $\mathscr{RI}(\Omega)$ is an inner function. Indeed, a $\theta$ of the form
$$
\theta(z) = ia + b \sum_{j=1}^{n} c_j \Gamma_{\alpha_j}(z)
$$
for some points $\alpha_j \in T_{\Omega}$, $a \in \mathbb{R}$, $b > 0$, $n\in \bN$ and $\sum_{j=1}^{n} c_j =1 \text{ with } c_j\geq 0$ extends meromorphically in a neighbourhood of $\overline{\Omega}$ with finitely many poles on $\partial \Omega$ since each $\Gamma_{\alpha_j}$ does so. Also, $\operatorname{Re}\theta(z) >0$ for each $z\in \Omega$ since so is $\operatorname{Re}\Gamma_{\alpha_j}(z)$ for each $j$. Furthermore, if $\zeta \in \partial \Omega$ is not a pole of $\theta$, then we have
	$$\operatorname{Re}\theta (\zeta)=  b \sum_{j=1}^{n} c_j \operatorname{Re}\Gamma_{\alpha_j}(\zeta)=0.$$
Thus $\frac{1-\theta}{1+\theta}$ is an inner function on $\Omega$ and it has removable singularities at the poles of $\theta$. Therefore $\frac{1-\theta}{1+\theta}$ is in $\cA(\Omega)$. In fact, $\mathscr{RI}(\bD)$ consists of rational inner functions. Our generalization of Carath\'eodory's result is as follows.	
\begin{thm} \label{CR-M}
		Any holomorphic function $\varphi:\Omega \to \bC$ satisfying $|\varphi(z)|\leq 1$ for all $z \in \Omega$ can be approximated (uniformly on compacta in $\Omega$) by inner functions in $ \mathscr{RI} (\Omega)$.
	\end{thm}

	\begin{proof}
	 Let $\varphi$ be as in the statement of the theorem.  Define a holomorphic map $\theta: \Omega\rightarrow \bC$ as follows
	\begin{align*}
		\theta(z)= \left(1-\varphi(z)\right)\left(1+\varphi(z)\right)^{-1}.
	\end{align*}
	Clearly, the real part of $\theta$ is non-negative. We shall show that the function $\theta$ can be approximated pointwise in $\Omega$ by Herglotz class functions $\theta_l$ satisfying the additional property of $\operatorname{Re}\theta_l(\zeta)=0$ for almost every $\zeta \in \partial \Omega$. Note that if we prove this, then for each $l\geq 1$ the functions $\varphi_l$ defined as
	$$\varphi_l(z)= \left(1-\theta_l(z)\right)\left(1+\theta_l(z)\right)^{-1}$$
	is an inner function and a simple calculation shows that the sequence $\{\varphi_l\}_{l\geq 1}$ converges to $\varphi$ uniformly on each compact subset of $\Omega$. By Theorem 1.1.21 in \cite{AHR}, which is Herglotz's representation obtained by the convexity argument mentioned above, there exists a finite positive regular Borel measure $\mu$ on $T_{\Omega}$ with total mass $\operatorname{Re}\theta(z_0)$ such that
	\begin{align*}
		\theta(z)=i\operatorname{Im} \theta(z_0)+ \int_{T_\Omega} \Gamma_\alpha(z)d\mu(\alpha).
	\end{align*}
	Since the Dirac measures at each point of $T_\Omega$ are the extreme points of the collection of regular Borel probability measures on $T_\Omega$, an application of Krein-Milman theorem enables us to get a sequence $\{\mu_l\}$ of probability measures such that 
	$$\mu_{l}\xrightarrow{weak*} \frac{1}{\operatorname{Re}\theta (z_0)} \mu, $$
	where $\mu_{l}$ is of the form $\sum_{j=1}^{n_l} c_j \delta_{\alpha_j}$ with $\alpha_j\in T_{\Omega}$, $c_j\geq 0$ and $\sum c_j=1$.
	Therefore, for each fixed $z\in \Omega$, 
	\begin{align*}
		\int_{T_{\Omega}} \Gamma_\alpha (z) d\mu_l \rightarrow \frac{1}{\operatorname{Re}\theta(z_0)}\int_{T_{\Omega}} \Gamma_\alpha (z) d\mu,
	\end{align*}
	and hence
	\begin{align*}
		\theta_l(z):= i\operatorname{Im}\theta(z_0)+ \int_{T_{\Omega}} \Gamma_\alpha (z) \operatorname{Re}\theta(z_0) d\mu_l \rightarrow \theta(z).
	\end{align*}
	Note that, $\theta_l$ is of the form 
	$$
	\theta_l(z)= i\operatorname{Im}\theta(z_0)+ \operatorname{Re}\theta(z_0) \sum_{j=1}^{n_l} c_j \Gamma_{\alpha_j}(z).
	$$
 So, $\varphi_l$ is in $\mathscr{RI}(\Omega)$. Also, $\|\varphi_l\|=1$ for each $l$ and $\varphi_l$ converges to $\varphi$ pointwise in $\Omega$. Hence $\varphi_l$ converges to $\varphi$ uniformly on each compact subset of $\Omega$.
\end{proof}
	
	\subsection{Properties of the Szeg\"o Kernel} This brief subsection will define the Szeg\"o kernel required for proving a Koranyi-Pukaszky formula. Certain properties of it are required which will be deduced below.
	
Let $\mu$ denote the usual normalized arc-length measure of the boundary of $\Omega$. Consider the Banach algebra $\cA(\Omega)$ of continuous functions on $\overline{\Omega}$ which are holomorphic in $\Omega$. Since $\cA(\Omega)$ is contained in $L^2(\mu)$, define the Hardy space of $\Omega$ as 
$$
H^2(\Omega, \mu)= \overline{\cA(\Omega)}^{L^2(\mu)}. 
$$
It is known that this a reproducing kernel Hilbert space. See \cite{M-S} where the reproducing kernel for the annulus is written. A Koranyi-Pukanszky formula for a Herglotz function using this kernel is not known. We need a {\em better} kernel.

 The {\em Hardy space} $H^2(\Omega, \omega_{z_0})$ is the closure of $\cA(\Omega)$ in $L^2(\omega_{z_0})$. It is interesting to note that as vector spaces $H^2(\Omega, \mu)$ and $H^2(\Omega, \omega_{z_0})$ are same. But they are different as reproducing kernel Hilbert spaces. For more details, see \cite{Abrahamse} as well as \cite{Ball-Clancey} where an explicit description of the reproducing kernel is given. Let $S_\Omega$ be the reproducing kernel for $H^2(\Omega, \omega_{z_0})$. The Szeg\"o kernel has the following properties:
 
	\begin{enumerate}
			\item[(i)] $S_\Omega (w,z_0)=1$ for every $w\in \Omega$.
			\item[(ii)] For each $z\in \Omega$, the function $s_z$ is in $\cA(\Omega)$, where $s_z (w)= S_\Omega (w,z)$ for $w\in \Omega$.
		\end{enumerate} 
		The {\em Poisson-Szeg\"o} kernel is defned as,
	$$
	\cP(z,\zeta):= \frac{|S_\Omega(z,\zeta)|^2}{S_\Omega (z,z)},\ \text{for}\ z\in \Omega\ \text{and}\ \zeta \in \partial \Omega.
	$$ We shall state two remarkable properties of $\cP$ which will be useful later.
	\begin{enumerate}
		\item[(a)]
		It is well known that, the Poisson-Szeg\"o kernel reproduces functions in $\cA(\Omega)$ i.e., for $f\in \cA(\Omega)$,
		$$
		f(z)= \int_{\partial \Omega} f(\zeta) \cP(z,\zeta) d\omega_{z_0}(\zeta), \text{ for all } z \in \Omega.$$
		\item[(b)]  $\cP(z_0, \zeta)=1$ and the real-valued function $\zeta \rightarrow \cP(z,\zeta)$ is continuous on $\partial \Omega$.
	\end{enumerate}

	\noindent \textbf{Notation:} Let us consider the subspace
	$$\cW(\Omega, \omega_{z_0}) = H^2(\Omega,\omega_{z_0}) + \overline{H^2(\Omega,\omega_{z_0})}= \left\lbrace f + \overline{g} : f, g\in H^2(\Omega, \omega_{z_0})   \right \rbrace.$$ 
	Then we can decompose $L^2(\omega_{z_0})$ as $L^2(\omega_{z_0})= \overline{\cW(\Omega, \omega_{z_0})}^{L^2(\omega_{z_0})} \oplus \cW(\Omega, \omega_{z_0})^\perp $. Measures, which can occur in integral representations for Herglotz class functions of several variables, have often been objects of studies, see for example \cite{L-N}. This is related to the quotienting alluded to in the Introduction. The class of measures defined below will feature in Theorem \ref{Her-multi}.
	\begin{definition} \label{deter}
		The real measures $\nu$ in $C(\partial \Omega)^*$ which satisfy $\int_{\partial \Omega} g d\nu = 0$ for any $g \in C(\partial \Omega) \cap \cW(\Omega,\omega_{z_0})^\perp$ will be called the determining measures and will be denoted by $\operatorname{RP}(\partial \Omega)$. 
	\end{definition}
\noindent	Note that the measures in $\operatorname{RP}(\partial \Omega)$ have a dependency on the measure $\omega_{z_0}$ related to the underlying Hardy space.
	It is clear that the set $\left \lbrace s_z, \overline{s_z} : z\in \Omega\right \rbrace$ is total in $\cW(\Omega, \omega_{z_0})$.
	
	\begin{lemma}\label{lemma-2}
		Fix $z\in \Omega$. Then the function $\psi(z,\zeta)= \cP(z,\zeta) -\left(s_z(\zeta) + \overline{s_z(\zeta)} -1 \right)$ is in $\cW(\Omega, \omega_{z_0})^\perp$. Moreover, $\psi(z,\cdot)$ is a real-valued function in $C(\partial \Omega)$.
	\end{lemma}
	\begin{proof}
		In view of  the above observation, it is enough to show that, $\langle \psi(z,\cdot), s_w \rangle = 0$ and $\langle \psi(z,\cdot), \overline{s_w} \rangle= 0$ for every $w\in \Omega$. Indeed, for $w\in \Omega$, using the reproducing property of $\cP$ and admissibility of the Szeg\"o kernel we have:
		\begin{align*}
			\left\langle \psi(z,\cdot), s_w  \right \rangle 
			&=\int_{\partial \Omega} \left[\cP(z, \zeta)- \left(s_z(\zeta) + \overline{s_z(\zeta)}-1\right)\right] \overline{s_w(\zeta)} d\omega_{z_0}(\zeta) \\
			&= \overline{s_w(z)} - \int_{\partial \Omega} \left[ s_z(\zeta)\overline{s_w(\zeta)} + \overline{s_z(\zeta) s_w(\zeta)} - \overline{s_w(\zeta)} \right] d\omega_{z_0}(\zeta) \\
			&= S_\Omega (w,z)-S_\Omega (w,z) -\langle s_{z_0}, s_z s_w \rangle  +  \langle s_{z_0}, s_w \rangle\\
			&=0.
		\end{align*}
		In a similar manner one can show that $\langle \psi(z,\cdot), \overline{s_w} \rangle =0$. Hence the proof.
	\end{proof}
	
	\subsection{The Koranyi-Pukanszky formula}
	It is clear from the previous lemma that $\cP(z,\zeta) = \operatorname{Re}(2S_\Omega (z,\zeta) -1) + \psi(z,\zeta)$ and that $2S_\Omega (z,\zeta) -1$ is holomorphic in the first argument in $\Omega$ and continuous in the second argument on $\partial \Omega$.

	\begin{thm}\label{Her-multi}
	Let $z_0 \in \Omega$. Let $\omega_{z_0}$ be the harmonic measure. Then the function
		\begin{align*}
			h(z)= \int_{\partial \Omega} (2S_\Omega (z,\zeta) -1) d\eta(\zeta), \ \text{for}\ z\in \Omega
		\end{align*} is holomorphic with non-negative real part for any determining  probability measure $\eta$.

Conversely, for every $f$ in $\mathfrak{H}_{z_0}(\Omega)= \{f\in \mathfrak{H}(\Omega): f(z_0)=1 \}$, there exists a regular Borel probability measure $\mu \in  \operatorname{RP}(\partial \Omega)$ such that
		\begin{align}\label{rep-multi}
			f(z)= \int_{\partial \Omega} \left(2S_\Omega (z, \zeta)-1 \right) d\mu(\zeta).
		\end{align}
	\end{thm}
	
		\begin{proof} The first part is the easier one. Since $2S_\Omega(z, \zeta)-1$ is continuous on $\partial \Omega$ in the second argument and the measure $\eta $ is finite, an application of dominated convergence shows that the function
		$$
		z \mapsto h(z)= \int_{\partial \Omega} \left[2S_\Omega(z, \zeta)-1 \right] d\eta (\zeta)
		$$ 
		is continuous in $\Omega $. Now, for any triangular path $\Delta$ in $\Omega $, an application of Cauchy's theorem along with Fubini's theorem yields
		\begin{align*}
			\int_{\Delta} \int_{\partial \Omega} \left[ 2S_\Omega(z, \zeta)-1\right] d\eta(\zeta) dz =0.
		\end{align*}  
		Therefore, by Morera's theorem $h$ is holomorphic in $\Omega$.
		Note that, for each $z\in \Omega$, \begin{align*}
			\operatorname{Re}h(z)&= \operatorname{Re}\int_{\partial \Omega} \left[ 2S_\Omega(z,\zeta)-1\right] d\eta(\zeta)\\
			&= \int_{\partial \Omega} \left(s_z(\zeta)+ \overline{s_z(\zeta)}-1 \right) d\eta(\zeta)\\
			&= \int_{\partial \Omega} \cP(z,\zeta) d\eta(\zeta) - \int_{\partial \Omega} \psi(z,\zeta) d\eta(\zeta) \\
			& \geq 0 \ (\text{by Lemma \ref{lemma-2} and since}\ \eta \in \operatorname{RP}(\partial \Omega) ). \end{align*}
In the converse direction, for an orthonormal basis $\{e_n\}_{n=1}^\infty$ of $H^2(\Omega, \omega_{z_0})$, we have
$$ S_\Omega (z,w)= \sum_{j=1}^\infty e_j(z)\overline{e_j(w)} \ \text{for}\ z, w\in \Omega.$$ It is holomorphic in the first variable and anti-holomorphic in the second variable. It is also important that for $w\in \Omega$, the kernel function $s_w:\Omega \to \mathbb{C}$ defined as $s_w(z)=S_\Omega(z,w)$ is in $\cA(\Omega)$. Note that, for every $f\in H^2(\Omega, \omega_{z_0})$,
		$$ f(z_0)= \langle f, s_{z_0} \rangle_{H^2(\Omega,\omega_{z_0})} = \langle f, 1\rangle_{H^2(\Omega, \omega_{z_0})} =\int_{\partial \Omega} f d\omega_{z_0}.$$ Therefore, $s_{z_0}\equiv 1$. Also, every $H^2(\Omega, \omega_{z_0})$ function has a non-tangential limit (boundary value) in $L^2(\omega_{z_0})$. Recall that the collection of real measures that are in $C(\partial \Omega)^*$ such that $\int_{\partial \Omega} g d\mu = 0$ for any $g \in C(\partial \Omega) \cap \cW(\Omega, \omega_{z_0})^\perp$ is denoted by $\operatorname{RP}(\partial \Omega)$. Moreover, in this case the measures in $\operatorname{RP}(\partial \Omega)$ can be described more precisely in the following way. The space $\cW(\Omega, \omega_{z_0})^\perp $ is the complex linear span of $m$ linearly independent functions $Q_1, \dots, Q_m$ in $C(\partial \Omega)$. See, Section 4.5 of Chapter 4 in \cite{Fisher-book}. Therefore a measure $\eta$ which represents a function in $\mathfrak{H}_{z_0}(\Omega)$ via the said integral formula, satisfies the conditions $$ \int_{\partial \Omega } Q_r d\eta =0 \ \text{for}\ r=1,\dots, m.$$

For $f\in \mathfrak{H}_{z_0}(\Omega)$, consider the bounded holomorphic function $g= \frac{1-f}{1+f}$ on $\Omega$. Then by our Theorem \ref{CR-M}, there exists a sequence of inner functions $\{g_n\}_{n=1}^\infty \subseteq \cA(\Omega)$ such that $g_n \to g $ uniformly on each compact subset of $\Omega$. Since $g(z_0)=0$, a closer look at the proof of Theorem \ref{CR-M} tells us that $g_n(z_0) =0$ (as they are the Cayley transforms of convex combinations of $\Gamma_{\alpha}$ and $\Gamma_{\alpha}(z_0) =1$). So, $\|(1-\frac{1}{n}) g_n \|_{\infty} <1$ and hence $$ f_n := \frac{1-(1-\frac{1}{n}) g_n}{1+(1-\frac{1}{n}) g_n}$$ is in $\cA(\Omega)$ with $f_n(z_0)=1$. Also, $f_n(z) \to f(z)$ for every $z\in \Omega$. It is easy to see that $\operatorname{Re}f_n(z) \geq 0$ for each $z\in \Omega$. Now we consider the sequence of regular Borel measures $\mu_n = \operatorname{Re}f_n d\omega_{z_0}$ on $\partial \Omega$. Then $\mu_n(\partial \Omega) = \operatorname{Re}f_n (z_0) =1.$ So, the sequence of non-negative measures $\mu_n$ is bounded and hence by Helly's selection principle, there exists a subsequence $\{\mu_{n_k}\}$ such that $\mu_{n_k} \xrightarrow{\text{weak*}} \nu$, for some regular Borel probability measure $\nu$ on $\partial \Omega$. Let $u \in C(\partial \Omega) \cap \cW(\Omega, \omega_{z_0})^\perp$. It is quite easy to check that
		$$  \int_{\partial \Omega} u d\mu_{n_k} =0 \ \text{for every}\ k.$$ By the weak* convergence of the sequence of measures, we get 
		$$
		\int_{\partial \Omega} u d\nu =0.
		$$
		Therefore $\nu \in \operatorname{RP}(\partial \Omega)$. Also for each $z\in \Omega $
		$$ \int_{\partial \Omega} \left[ 2S_\Omega(z,\zeta)-1\right] d\mu_{n_k}(\zeta) \rightarrow \int_{\partial \Omega} \left[ S_\Omega(z,\zeta)-1\right] d\nu(\zeta) = h(z) \ (\text{say}). $$
		Further,
		\begin{align*}
			\operatorname{Re} \int_{\partial \Omega} \left[ S_\Omega(z,\zeta)-1\right] d\mu_{n_k}(\zeta) 
			&= \int_{\partial \Omega} \left( \cP(z,\zeta)-\psi(z,\zeta) \right) \operatorname{Re}f_{n_k}(\zeta) d\omega_{z_0}(\zeta)\\
			&=\operatorname{Re} \int_{\partial \Omega} \cP(z,\zeta) f_{n_k}(\zeta) d\omega_{z_0}(\zeta) - \langle \psi(z,\cdot), \operatorname{Re}f_{n_k}\rangle_{L^2(\partial\Omega, \omega_{z_0})} \\
			&=\operatorname{Re} f_{n_k}(z) \ \left(\text{since} \ \mu_{n_k} \in \operatorname{RP}(\partial \Omega)\right)\\
			& \rightarrow \operatorname{Re} f(z) \ \text{as}\ k \to \infty.
		\end{align*} 
		Therefore, $\operatorname{Re}f(z)= \operatorname{Re} h(z)$, for every $z\in \Omega$. Also, note that $f(z_0)=1= h(z_0)$. Hence, $f(z)=h(z)$ for all $z\in \Omega.$ This proves the integral representation \eqref{rep-multi} for a function in $\mathfrak{H}_{z_0}(\Omega)$.

	\end{proof}

\subsection{The equivalence}

 Theorem \ref{CR-M} implies Theorem 1.1.21 in \cite{AHR}, showing that Carath\'eodory's approximation is equivalent to Herglotz's representation obtained via convexity. 
		Indeed, consider  a holomorphic map $f : \Omega \rightarrow \mathbb{C}$ such that $\operatorname{Re} f(z)\geq 0$ for all $z \in \Omega$.  Assume, without loss of generality, that  $f(z_0)=1$.
		The Cayley transform $\theta= \left(1-f\right)\left(1+f\right)^{-1}$ of $f$ is holomorphic in $\Omega$ and $|\theta(z)|\leq 1$ for all $z\in\Omega$. Also, $\theta(z_0)=0$. Invoke Carath\'eodory's Theorem \ref{CR-M} to get a sequence of inner functions $\{\theta_k\}$ in $\mathscr{RI}(\Omega)$ with $\theta_k(z_0)=0$ such that $\theta_k\rightarrow \theta$ uniformly on compacta in $\Omega$. Let $\psi_k$ denote the inverse Cayley transform of $\theta_k$. Recall the structure of functions in $\mathscr{RI}(\Omega)$ from \eqref{Inner-Mul}. It immediately follows that 
	
		$$
		\psi_k = \sum_{j=1}^{n(k)} \lambda_j^{(k)} \Gamma_{\alpha_j^{(k)}}. 
		$$
		for some positive integer $n(k)$, finitely many points $\alpha_j^{(k)}$ in $T_{\Omega}$ for $j=1,\dots, n(k)$  and non-negative scalars $\lambda_j^{(k)}$ such that $\sum_{j=1}^{n(k)} \lambda_j^{(k)} =1 $.
		
		Let $\mu_k$ be the measure $\sum_{j=1}^{n(k)} \lambda_j^{(k)} \delta_{\alpha_j^{(k)}}$ on $T_\Omega$.  Since $\{\mu_k\}_{k\geq 1}$ is a sequence of regular Borel probability measures on $T_{\Omega}$, we can extract a subsequence $\{\mu_{k_l}\}$ such that $\mu_{k_l} \xrightarrow{\text{weak}^*} \nu$  for some regular Borel probability measure $\nu$ on $T_{\Omega}$. Therefore 
		\begin{align*}
			\psi_{k_l} (z)= \int_{T_{\Omega}} \Gamma_\alpha (z) d\mu_{k_l} (\alpha) \rightarrow \int_{T_{\Omega}} \Gamma_\alpha (z) d\nu (\alpha)
		\end{align*}
by continuity of $ \Gamma_\alpha$ in $\alpha$ (as a consequence of Lemma 1.2.20 in \cite{AHR}). But $\psi_{k_l}(z)$ converges to $f(z)$ for each $z$. Hence we have 
		$
		f(z)= \int_{T_{\Omega}} \Gamma_\alpha (z) d\nu (\alpha).
		$
	
	\subsection{ Carath\'eodory's theorem for matrix-valued functions}

\emph{There is a crucial difference between  Herglotz's representations of $N\times N$ matrix-valued holomorphic functions with non-negative real part on the unit disc and on a multi-connected domain}. 

Recall from \eqref{theta} that in the former case, the integrands are scalar-valued functions (extreme points of $\mathfrak H_{0}(\mathbb D)$) and the measure is positive matrix-valued. In the latter case, the integrand functions are matrix-valued. They are the extreme points of $\mathfrak H_{z_0}(\Omega, \mathcal B (\mathbb C^N))$ whereas the measure is a probability measure, see \cite[Theorem 3.5]{BH}. We shall prove Carath\'eodory's approximation result for matrix-valued functions on multi-connected domains using the results of Section 3 of \cite{BH}. 

Let $T_\Omega^N$ be the union of all points $(\mathbf{u},\mathbf{W})$ where $\mathbf{u}=(u_1,\dots, u_n)  \text{ and } \mathbf{W}=(W_1, \dots, W_n)$ for some natural number $m+1 \le n \le (m+1)N^2$ satisfying $u_1, \dots, u_n$ are $n$ distinct points in $\partial \Omega$ and $W_1, \dots, W_n$ are $N\times N$ non-negative definite matrices satisfying properties (iii) and (iv) in Theorem 3.3 of \cite{BH}. These properties stem from the characterization of the extreme points of a certain class of positive matrix-valued measures. A satisfactory explanation of them will require more space than we want to spend. Instead, we refer the interested reader to the excellently written Subsection 3.2 of \cite{BH}. It is known from \cite{BH} that the extreme points $F_{\mathbf{u},\mathbf{W}}$ of the Herglotz class $\mathfrak H_{z_0}(\Omega, \mathcal B (\mathbb C^N))$ are indexed by $T_\Omega^N$. 

\begin{thm}[ Carath\'eodory's approximation for matrix-valued functions]
	Let $\theta : \Omega \rightarrow M_N(\mathbb{C})$ be a holomorphic function on $\Omega$ such that $\|\theta(z)\| \leq 1$ for every $z\in \Omega$. Then $\theta$ can be approximated (uniformly on compacta in $\Omega$) by a sequence of inner functions.
\end{thm} 
	\begin{proof} Here is a sketch of the proof. We deal only with the case when $\|\theta(z)\|<1$ for all $z \in\Omega $. Let $f$ be the Cayley transform of $\theta$, i.e., 
		$$
		f(z)= (I-\theta(z))(I+ \theta(z))^{-1}. 
		$$
		By an analysis similar to what we did earlier, we can take $f$ to be in a normalized form. Since the set of extreme points of the compact convex set $\mathfrak H_{z_0}(\Omega, \mathcal B (\mathbb C^n))$ (in the compact-open topology which is metrizable)  is 
$$ \left\lbrace F_{\mathbf{u}, \mathbf{W}}: (\mathbf{u}, \mathbf{W})\in T_{\Omega}^N \right \rbrace ,
	$$
the Krein-Milman theorem implies that the function $f$ can be approximated uniformly on each compact subset of $\Omega$ by a sequence of $N\times N$ matrix-valued functions $\{f_n\}$ which are convex combinations of the form
	\begin{align}\label{structure}
		f_n = \sum_{j=1}^{r(n)} \lambda_j^{(n)} F_{\mathbf{u}_j^{(n)}, \mathbf{W}_j^{(n)}} \text{ such that } \sum_{j=1}^{r(n)} \lambda_j^{(n)} =1 \text{ and } \lambda_j^{(n)} \geq 0
	\end{align}
for certain points $(\mathbf{u}_j^{(n)}, \mathbf{W}_j^{(n)})$ in $T_\Omega^N$.  By the very nature of $ F_{\mathbf{u}, \mathbf{W}}$, it satisfies $ \operatorname{Re} F_{\mathbf{u}, \mathbf{W}} (\zeta) = 0$ for $\zeta$ in $\partial \Omega$ except for a finitely many points. Thus, we have $ \operatorname{Re} f_n(\zeta)= 0$ for $\zeta$ in $\partial \Omega$ except for a finitely many points. 

		Taking the inverse Cayley transform, we see that the matrix-valued functions
		$$
		\theta_n = (I-f_n)(I+f_n)^{-1}
		$$
		 are inner and converge to $ \theta$  uniformly on compacta in $\Omega$. 
	\end{proof}
	
	\section{Integral representation for operator-valued Herglotz functions} \label{Opvalued}

\subsection{Uniqueness of  Herglotz's representation in scalar case:}
In this subsection we investigate the uniqueness of the determining measures in the Koranyi-Pukanszky type formula for a special class of complete Reinhardt domains {\em without any assumption on the regularity of the boundary}. The uniqueness of the measure is important because it always leads to Herglotz's representation for operator-valued functions. More details on how it is done is in Theorem \ref{operator-valued}.
	
Let $\Omega$ be a bounded, strongly star-like Reinhardt domain centred at 0 in $\bC^d$. Suppose $\sigma$ is a finite regular positive Borel measure on $\partial \Omega$ with the support $F_\sigma$ not contained in the union of the coordinate planes and it is a {\em $\bT^d$-invariant} measure i.e.,  
of the form: $$d\sigma(\mathbf{z}) = d\mu(|z_1|, \dots, |z_d|) dm, \ \mathbf{z}=(z_1, \dots, z_d)$$ where $\mu$ is supported on $\{(|w_1|,\dots, |w_d|): \mathbf{w}=(w_1, \dots, w_d)\in \partial \Omega \}$ and $m$ is the normalized Haar measure on $\bT^d$, see \cite{Curto}. We also assume that $\sigma$ is {\em massive} on the Shilov boundary $\partial_S(\Omega)$ that is, for any Borel subset $E \subseteq \partial \Omega$ if $\sigma(E)=0$ then 
$$ \overline{\partial \Omega\setminus E} \supset \partial_S\Omega.$$

\textbf{Claim:} $\partial_S\Omega \subseteq F_\sigma$.

\noindent To prove the claim, we assume on the contrary that there exists $\boldsymbol{\zeta} \in \partial_S\Omega \setminus F_\sigma $. Then there must exist an open set $N_{\boldsymbol{\zeta}}$ in $\partial \Omega$ containing $\boldsymbol\zeta$ such that $\sigma(N_{\boldsymbol{\zeta}})=0$. Thus we have $\partial \Omega\setminus N_{\boldsymbol{\zeta}}= \overline{\partial \Omega\setminus N_{\boldsymbol{\zeta}}} \supset \partial_S\Omega $ and this is not possible since $\boldsymbol{\zeta} \in \partial_S\Omega$.

 Now, consider the holomorphic hull of $F_\sigma$, viz.,
$$
\widehat{F_\sigma}= \left\lbrace  \textbf{z}\in \bC^d: |g(\textbf{z})|\le \sup_{\textbf{w}\in F_\sigma} |g(\textbf{w})|  \text{ for all entire functions }\ g \right\rbrace.
$$
In a similar way one can define $\widehat{\partial_S\Omega}$ and then it is clear from the above claim that 
$$\overline{\Omega} \subseteq \widehat{\partial_S\Omega} \subseteq \widehat{F_\sigma} . $$

Denote $\hat{\sigma}(\alpha, \beta) = \int \textbf{z}^\alpha \bar{\textbf{z}}^\beta d\nu(\textbf{z})$ for $(\alpha, \beta) \in \mathbb{Z}_{+}^d \times \mathbb{Z}_{+}^d$. We quote two lemmas from \cite{Kolaski}.
\begin{lemma}{(\cite[Lemma 2.1.4]{Kolaski})}
The domain of convergence of the power series $\sum_{\alpha \in \mathbb{Z}_{+}^d} \frac{\mathbf{z}^{2\alpha}}{\hat{\sigma}(\alpha, \alpha)}$ is precisely the interior $(\widehat{F_\sigma})^\circ$.
\end{lemma}

\begin{lemma}{(\cite[Lemma 2.1.5]{Kolaski})}
The series \begin{align*}
	S_{\Omega, \sigma}(\mathbf{z},\mathbf{w})= \sum_{\alpha \in \mathbb{Z}_{+}^d} \frac{\mathbf{z}^\alpha \bar{\mathbf{w}}^\alpha}{\hat{\sigma}(\alpha, \alpha)}
\end{align*}
converges absolutely and uniformly on $A \times \widehat{F_\sigma} $ for any compact subset $A$ of $(\widehat{F_\sigma})^\circ$.
\end{lemma}

Note that we have $\overline{\Omega} \subseteq \widehat{F_\sigma}$ and hence $\Omega \subseteq (\widehat{F_\sigma})^\circ.$ Hence the series representing $S_{\Omega, \sigma}$ converges absolutely and uniformly on compact subsets of $A \times \overline{\Omega}$ for any compact subset $A$. The existence of such a function (referred to as the Szeg\"o kernel) and their continuity on the boundary appeared in \cite{Dautov} but we could not find a proof and that is why we mention the brief proof through the lemmas above. A Koranyi-Pukanszky type representation of a Herglotz function $f$ on $\Omega$ as the following
\begin{align}\label{KP-starlike}
f(\mathbf{z})=i \operatorname{Im}f(0)+ \int_{F_\sigma} (2S_{\Omega, \sigma}(\mathbf{z}, \boldsymbol{\zeta}) -1 ) d\nu(\boldsymbol{\zeta})
\end{align}
has been obtained in \cite{Dautov}. The measure $\nu$ is a finite regular non-negative Borel measure on $F_\sigma$ such that $$\int_{F_\sigma} g d\nu = 0$$ for every continuous function $g$ in $L^2(\sigma)\ominus [\cA(\Omega)+ \overline{\cA(\Omega)}]$ where $\overline{\cA(\Omega)}$ denotes the set of complex conjugates of the functions in $\cA(\Omega)$. We denote the collection of all such measures by $\operatorname{RP}(F_\sigma)$.

\begin{thm}
The measure $\nu$ in $\operatorname{RP}(F_\sigma)$ representing the Herglotz function $f$ in $\Omega$ via the integral representation \eqref{KP-starlike} is unique.
\end{thm}

\begin{proof}
 Indeed, if there are two such measures $\mu_1$ and $\mu_2$ for which \eqref{KP-starlike} holds, then 
\begin{align}
\int_{F_\sigma} (2S_{\Omega, \sigma}(\mathbf{z}, \boldsymbol{\zeta}) -1 ) d(\mu_1 -\mu_2)(\boldsymbol{\zeta})=0 \ \text{for every}\ \mathbf{z}\in \Omega.
\end{align}
Taking real parts into consideration we get

\begin{align}\label{unique}
\int_{F_\sigma} (2 \operatorname{Re}S_{\Omega, \sigma}(\mathbf{z}, \boldsymbol{\zeta}) -1 ) d(\mu_1 -\mu_2)(\boldsymbol{\zeta})=0 \ \text{for every}\ \mathbf{z}\in \Omega.
\end{align}

Now, an argument similar to Lemma \ref{lemma-2} implies that for every $\mathbf{z} \in \Omega$ the function 

$$\boldsymbol{\zeta} \ \xrightarrow{g_\mathbf{z}}\ \frac{|S_{\Omega, \sigma}(\mathbf{z}, \boldsymbol{\zeta})|^2}{S_{\Omega, \sigma}(\mathbf{z}, \mathbf{z})} -\left(S_{\Omega,\sigma}(\mathbf{z}, \boldsymbol{\zeta}) + \overline{S_{\Omega,\sigma}(\mathbf{z}, \boldsymbol{\zeta})} -1 \right)$$
 is continuous on $\partial \Omega$ and $\int_{F_\sigma} g_{\mathbf{z}} d\mu_j =0$ for $j=1,2.$ This observation along with \eqref{unique} implies the following for every $\mathbf{z} \in \Omega$,

\begin{align*}
&\int_{F_\sigma} \frac{|S_{\Omega, \sigma}(\mathbf{z}, \boldsymbol{\zeta})|^2}{S_{\Omega, \sigma}(\mathbf{z}, \mathbf{z})} d(\mu_1 -\mu_2)(\boldsymbol{\zeta}) = 0 \\
\text{ i.e., } & \int_{F_\sigma} |S_{\Omega,\sigma}(\mathbf{z}, \boldsymbol{\zeta})|^2 d(\mu_1 -\mu_2)(\boldsymbol{\zeta}) =0 \\
\text{ i.e., } & \sum_{\alpha, \beta \in \mathbb{Z}_{+}^d} [\hat{\sigma}(\alpha, \alpha) \hat{\sigma}(\beta, \beta)]^{-1} \mathbf{z}^\alpha \bar{\mathbf{z}}^\beta \int_{F_\sigma}  \bar{\boldsymbol{\zeta}}^\alpha \boldsymbol{\zeta}^\beta d(\mu_1 -\mu_2) = 0 \\
\text{ i.e., } & \int_{F_\sigma} \bar{\boldsymbol{\zeta}}^\alpha \boldsymbol{\zeta}^\beta d(\mu_1 -\mu_2) =0 \ \text{for all}\ \alpha, \beta \in \mathbb{Z}_{+}^d.
\end{align*}
Therefore, $\mu_1- \mu_2$ annihilates all the polynomials in $\boldsymbol{\zeta}$ and $\bar{\boldsymbol{\zeta}}$ and hence by the Stone-Weierstrass theorem $\mu_1 =\mu_2$.

\end{proof}

In view of this large class of domains where the measure appearing in Herglotz's representation is unique, we have the following.

\subsection{Extension of Herglotz's representation to operator-valued case:}	
	
	There is a general scheme to obtain integral representations for operator-valued Herglotz functions whenever we have a unique representation for each scalar-valued function with respect to some fixed integral kernel. 
	\begin{thm} \label{operator-valued}
		 Let $\Omega$ be a bounded domain. Let $T$ be a closed subset of $\partial \Omega$. Let  $\mathbf{z}_0$ be a fixed point in $\Omega$. Suppose there is a function $H$, holomorphic in the first variable, continuous in the second variable on $\Omega \times T$ and $H(\mathbf{z}_0, \cdot)=1$ such that for every function $f$ in $\mathfrak{H}(\Omega)$, there is a unique non-negative Borel measure $\nu$ on $T$ satisfying an integral representation, i.e., 
\begin{align} \label{unique-rep} f(\mathbf{z})= i\operatorname{Im}f(\mathbf{z}_0)+ \int_{T} H(\mathbf{z},\boldsymbol{\zeta}) d\nu(\boldsymbol{\zeta}) \text{ for each } \mathbf{z} \text{ in } \Omega. \end{align}
Then a $\cB(\cH)$-valued function $f$ is in $\mathfrak{H}\left(\Omega, \cB(\mathcal{H}) \right)$  if and only if there exists a unique non-negative $\cB(\cH)$-valued bounded regular Borel measure $E$ on $T$ such that 
		\begin{align}\label{matrix-valued-rep}
			f(\mathbf{z})= i\operatorname{Im}f(\mathbf{z}_0)+ \int_{T} H(\mathbf{z},\boldsymbol{\zeta}) dE(\boldsymbol{\zeta})  \text{ for each } \mathbf{z} \text{ in } \Omega, 
		\end{align}
		 Moreover, if $M(\Omega, T)$ denotes  the class of all non-negative Borel measures $\nu$ on $T$ which give rise to Herglotz class functions through formula \eqref{unique-rep}, then the non-negative measure $\langle E(\Delta)h, h \rangle$ defined on Borel subsets $\Delta$ of $T$ is in $M(\Omega, T)$ for every non-zero $h$ in $\mathcal H$.
	\end{thm} 
	\begin{proof}
		If we have a function $f$ satisfying equation \eqref{matrix-valued-rep}, then it is easy to see that $f$ is holomorphic. Also, 
		$$\langle f(\mathbf{z}) h, h\rangle =i\operatorname{Im}f(\mathbf{z}_0)+  \int_{T} H(\mathbf{z},\boldsymbol{\zeta}) dE_{h,h}(\boldsymbol{\zeta}).$$ Therefore, the existence of the integral representation for scalar-valued functions proves that $f\in \mathfrak{H}(\Omega, \cB(\cH))$. 
		
		Conversely, for $f \in \mathfrak{H}\left(\Omega, \cB(\mathcal{H}) \right)$, define for each $h\in \cH$, $f_h: \Omega \to \bC $ by $$f_h(\mathbf{z})=\langle f(\mathbf{z})h, h \rangle.$$ 
		Then, the integral representation for scalar-valued functions implies that there exists a unique non-negative measure $\mu_h$ such that
		$$
		f_h(\mathbf{z})= i\operatorname{Im}f_h(\mathbf{z}) + \int_{T} H(\mathbf{z},\boldsymbol{\zeta})d\mu_h(\boldsymbol{\zeta})
		$$
		with $\mu_h(T)= \langle \operatorname{Re}f(\mathbf{z}_0) h, h \rangle$.
		Let us define a $\cB(\cH)$-valued measure $E$ on the Borel sigma algebra of $T$ as follows:
		\begin{align}\label{quadratic-form}
			\left\langle E(\Delta)h, h \right\rangle = \mu_h(\Delta) 
		\end{align}
		for any Borel set $\Delta \subset T$. Then for $\alpha \in \bC$, $$ \left\langle E(\Delta)(\alpha h), \alpha h \right\rangle = \mu_{\alpha h}(\Delta). $$ Note that the measure $\mu_{\alpha h}$, represents the function $$\mathbf{z} \mapsto \langle f(\mathbf{z})(\alpha h), \alpha h \rangle = |\alpha|^2 f_h(\mathbf{z}).$$ Therefore, by the uniqueness of the measure in the integral representation for the scalar case we have $$\mu_{\alpha h} = |\alpha|^2 \mu_h.$$ So, the function $\cQ_{\Delta}: \cH \to \bC $ defined as
		$$ \cQ_{\Delta}(h)=\left\langle E(\Delta)h, h \right\rangle \ \text{for}\ h\in \cH,$$ is a quadratic form with $|\cQ_{\Delta}(h)|\leq \|\operatorname{Re}f(\mathbf{z}_0)\| \|h\|^2$. Therefore polarization of $\cQ_{\Delta}$ induces a bounded sesquilinear form on $\cH \times \cH$ which in turn provides a bounded operator on $\cH$, call it $E(\Delta)$ again. This justifies why we have claimed $E(\Delta)$ to be a bounded operator in the definition of $E$. Positivity of $E(\Delta)$ is obvious from $\eqref{quadratic-form}$. Now it is routine to check that $E$ is a regular Borel measure with 
		$$\operatorname{Sup}\left\lbrace \|E(\Delta)\|: \Delta \subseteq_{\text{Borel}} T \right \rbrace = \|\operatorname{Re}f(\mathbf{z}_0)\| <\infty.$$ 
		
		\textbf{Claim:} The function $f$ is of the form \eqref{matrix-valued-rep}, for the $\cB(\cH)$-valued measure $E$.
		
		To prove the claim, take any $h\in \cH$. Then
		\begin{align*}
			\left \langle \left[i \operatorname{Im}f(\mathbf{z}_0) + \int_{T} H(\mathbf{z},\boldsymbol{\zeta}) dE(\boldsymbol{\zeta})\right] h, h \right \rangle 
			&= i \operatorname{Im}\left \langle f(\mathbf{z}_0) h, h \right \rangle +  \int_{T} H(\mathbf{z}, \boldsymbol{\zeta})dE_{h,h} \\
			&= f_h(\mathbf{z})= \langle f(\mathbf{z})h, h\rangle.
		\end{align*}
		Therefore by polarization, we have a representation of $f$ as in \eqref{matrix-valued-rep}. Our construction of the positive operator-valued measure $E$ justifies the last sentence of the statement of the theorem. The uniqueness of the measure $E$ follows from the uniqueness of the measures $E_{h,h}=\mu_h$.
	\end{proof}
	
	\section{An exceptional domain}
	For $d\geq 2$, let $s_1,\dots, s_d$ be the elementary symmetric polynomials in $d$ variables. Define $s: \bD^d \rightarrow \bC^d$ by 
	$$ s(\mathbf{z})=\left( s_1(\mathbf{z}), \dots, s_d(\mathbf{z})\right) .$$
	Let $J$ be the complex Jacobian of this map. Let $m$ be the normalized Haar measure on the torus $\bT^d$. The symmetrized polydisc $\bG_d := s(\bD^d)$ is non-convex, is not star-like for $d\geq 3$, and does not have a $C^2$-smooth boundary, see \cite{Jar-Pflug}. Nevertheless, this domain has a Szeg\"o kernel 
	$$ S_{\bG_d}(s(\mathbf{z}), s(\mathbf{w}))= \prod_{i,j=1}^d (1-z_i \bar{w}_j)^{-1} \text{ for } \mathbf{z},\mathbf{w}\in \bD^d. $$ See \cite{MRZ}. The corresponding reproducing kernel Hilbert space $H^2(\bG_d)$ consists of holomorphic functions on $\bG_d$ and can be isometrically embedded into $L^2(b\bG_d, \nu)$ where 
	$$ b\bG_d = \{s(\mathbf{z}): \mathbf{z}\in \bT^d \}$$ is the distinguished boundary of $\bG_d$ and $\nu$ is the Borel measure on $b\bG_d$ obtained as the push-forward of the measure $|J|^2 dm$ on $\bT^d$. Using the Poisson-Szeg\"o kernel for $\bG_d$ and a scaling technique, it can be shown that $f \in \mathfrak{H}_0 (\bG_d)$ if and only if there is a probability measure $\eta$ on $\bG_d$ satisfying 
	$$
	\int_{b\bG_d} g d\eta =0 \text{ for all } g\in C(b\bG_d) \cap \left(H^2(\bG_d) + \overline{H^2(\bG_d)}  \right)^\perp 
	$$ 
	such that 
	
$$f(\mathbf s)= \int_{b\bG_d} \left[2S_{\bG_d}(\mathbf{s}, \mathbf{t})-1 \right] d\eta(\mathbf{t}) $$
	for every $\mathbf s$ from $\bG_d$. 
	
	Carath\'eodory's approximation theorem for the symmetrized polydisc has been found in \cite{BK}. Whether Herglotz's representation and Carath\'eodory's approximation are equivalent on the polydisc or on the symmetrized polydisc remains an open problem.


\end{document}